\documentclass[12pt]{amsart}

\usepackage{mathtools}

\usepackage{hyperref}
\usepackage{multimedia}
\usepackage{amsmath}
\usepackage{amsfonts}
\usepackage{amssymb}
\usepackage{amsthm}
\usepackage{latexsym}
\usepackage{color}
\usepackage{enumerate}
\usepackage{mathrsfs}
\usepackage{comment}
\usepackage{setspace}
\usepackage{appendix}
\usepackage{tikz} \usetikzlibrary{arrows,shapes} % to draw diagrams

\usepackage{graphics}
\usepackage{graphicx}
 \usepackage{subfigure}
\usepackage{float}
\newtheorem{definition}{Definition}[section]
\newtheorem{lemma}{Lemma}[section]
\newtheorem{thm}{Theorem}[section]
\newtheorem{prop}{Proposition}[section]

\theoremstyle{remark}
\newtheorem{remark}{Remark}[section]

\numberwithin{equation}{section}

\def\tr{\textmd{tr}}

\def\div{\mathrm{div}}

\def\R{\mathbb{R}}

\def\R{\mathbb{R}}

\def\vh{\vspace{.2cm}}

\def\p{\partial}
\def\So{\Sigma_o}
\def\Sh{\Sigma_h}
\def\hOmega{\hat{\Omega}}

\def\mh{\mathfrak{m}_{_H}}

\def\W{\mathcal{W}}

\newcommand{\be}{\begin{equation}}
\newcommand{\ee}{\end{equation}}
\newcommand{\bee}{\begin{equation*}}
\newcommand{\eee}{\end{equation*}}

\begin{document}

\title{On compact $3$-manifolds with nonnegative scalar curvature with a CMC boundary component}

\author{Pengzi Miao}
\address[Pengzi Miao]{Department of Mathematics, University of Miami, Coral Gables, FL 33146, USA.}
\email{pengzim@math.miami.edu}
%\thanks{$^1$Research partially supported by Simons Foundation Collaboration Grant for Mathematicians \#281105.}

\author{Naqing Xie}
\address[Naqing Xie]{School of Mathematical Sciences, Fudan
University, Shanghai 200433, China.}
\email{nqxie@fudan.edu.cn}
%\thanks{$^2$Research partially supported by the National Science Foundation of China \#11671089, \#11421061.}



%\thanks{The first named author's research was partially supported by Simons Foundation Collaboration Grant for Mathematicians \#281105.  The second named author's research was partially supported by the National Science Foundation of China \#11671089, \#11421061.}

\begin{abstract}
We apply the Riemannian Penrose inequality and the Riemannian positive mass theorem 
 to  derive  inequalities on the boundary of a class of compact Riemannian  $3$-manifolds with nonnegative scalar curvature.
The boundary of such a manifold has a  CMC component, i.e. a $2$-sphere with positive constant mean curvature;
and the rest of the boundary, if nonempty, consists of  closed minimal surfaces. 
A key  step in our proof  is  the construction of a collar extension  that  is 
inspired by the method of Mantoulidis-Schoen \cite{M-S}.
\end{abstract}

\maketitle

\markboth{Pengzi Miao and Naqing Xie}{$3$-manifolds with a CMC boundary component}

\section{Introduction and statement of results}
In this paper, we are interested in a  compact Riemannian $3$-manifold $\Omega$ with 
nonnegative scalar curvature, with boundary $ \p \Omega$, such that 
$\p \Omega $ has a component $ \So$ that is  a topological  $2$-sphere with positive  mean curvature.
When $ \p \Omega \setminus \So \neq \emptyset $, we  assume that 
$\p \Omega \setminus \So $ is the unique, closed minimal surface (possibly disconnected) in $\Omega$, i.e. 
there are no other closed minimal surfaces in $\Omega$. 
In this case, we denote $ \p \Omega \setminus \So $ by $ \Sh$. 
In a relativistic context, such an $\Omega$ represents a finite body in a time-symmetric  initial data set, surrounding 
the apparent horizon modeled by  $\Sh$. 

Motivated by the  quasi-local mass problem (cf. \cite{Penrose-qlmass}), 
we want to understand the effect of nonnegative scalar curvature and the existence of $\Sh$ on the boundary 
geometry of $\So$. To be more precise, let $g$ denote the induced metric on $ \So$ and $H$ be the mean curvature 
of $ \So$ in $\Omega$. We want to understand the restriction imposed by the scalar curvature and the horizon boundary 
$\Sh$ on the pair $(g, H)$. 

A special case of this question  was studied in \cite{Miao09}. It was proved in \cite{Miao09} that 
\bee
(\So, g) \ \mathrm{is \ a \ round \ sphere}  \Rightarrow
 \sqrt{ \frac{ | \So | }{ 16 \pi } }
\left[ 1 - \frac{ 1 }{ 1 6 \pi | \So | } \left( \int_{\So} H  d \sigma
\right)^2 \right] \geq \sqrt{ \frac{ | \Sh | }{ 16 \pi } }  ,
\eee
where  $ | \So |$, $ | \Sh|$ are the area of $\So$, $ \Sh$, respectively, and $ d \sigma $ denotes  the area element on $ \So$.
The left side of the above inequality  closely resembles the Hawking mass \cite{Hawking}
of $ \So$ in $ \Omega$, given by
\bee
\mh ( \So ) = \sqrt{ \frac{ | \So | }{ 16 \pi } }
\left[ 1 - \frac{ 1 }{ 1 6 \pi  }  \int_{\So} H^2 \ d \sigma  \right] .
\eee
The Hawking mass
functional $ \mh  (\cdot) $ played a key role in Huisken and Ilmanen's proof
of  the Riemannian Penrose inequality (cf. \cite{Bray01, H-I01})
 when the horizon is connected. In particular,  by the results in \cite{H-I01},
if a weak solution  $ \{ \Sigma_t \}$ consisting of connected surfaces
to the inverse mean curvature flow with initial condition $\Sh$ exists in $ \Omega$
and  if  $ \So$ happens to be a leaf  in $ \{ \Sigma_t \}$,
then one would have  $ \mh (\So ) \ge \sqrt{ \frac{ | \Sh |}{16 \pi} } $.

In general, without imposing suitable conditions on $\So$, 
one should not expect to have  $ \mh (\So ) \ge \sqrt{ \frac{ | \Sh |}{16 \pi} } $ since $ \mh (\So)$ 
may even fail to be positive. 
On the other hand,  if a $2$-surface is a stable  constant mean curvature (CMC) surface 
in a $3$-manifold with nonnegative scalar curvature, 
Christodoulou and Yau \cite{C-Y}  showed  that its Hawking mass is always nonnegative.

In this paper, we consider an  $\Omega$ in which $ \So$ is a
CMC surface. We have

\begin{thm} \label{thm-main-1}
Let $ \Omega$ be a compact, orientable, Riemannian $3$-manifold with boundary $ \p \Omega$.
Suppose $ \p \Omega$ is  the disjoint union of  $ \Sigma_o $ 
and $\Sh$ such that 
\begin{itemize}
\item[(a)] $\Sigma_o$ is  a topological $2$-sphere with constant mean curvature $H_o > 0$; 
\item[(b)] $\Sh$, which may have multiple components,  is a minimal surface; and
\item[(c)] there are no other closed minimal surfaces in $\Omega$.
\end{itemize}
Suppose  $\Omega $ has nonnegative scalar curvature and  the induced metric $g$ on $ \Sigma_o$ 
 has positive Gauss curvature.
 There exists  a quantity $ 0 <  \eta (g) \le \infty$,  uniquely determined   by $(\Sigma_o, g)$ and 
 invariant under   scaling of $g$, such that if
\bee 
 \mathcal{W} : =  \frac{1}{16 \pi} \int_{\Sigma_o} H_o^2 d \sigma  < \eta (g), 
 \eee
then
\be \label{eq-main-1}
 \sqrt{ \frac{ | \Sh |}{16 \pi} } \le  
\left[    \frac{    \mathcal{W} } {  \eta (g) -   \mathcal{W} }  \right]^\frac12  \sqrt{ \frac{ | \Sigma_o  |}{16 \pi} }
+ \mh (\So)  .
\ee
Here  $ \eta (g) = \infty  $ if  $g$ is a round metric. In this case,  \eqref{eq-main-1}  reduces to
$ \sqrt{ \frac{ | \Sh |}{16 \pi} }  \le  \mh (\Sigma_o)  $.
\end{thm}

Theorem \ref{thm-main-1}  has  the following analogue when $ \p \Omega = \So$. 

\begin{thm} \label{thm-main-1-1}
Let $ \Omega$ be a compact, Riemannian $3$-manifold with nonnegative scalar curvature, 
with boundary $ \So$.
Suppose $ \So $ is a topological $2$-sphere with constant mean curvature $H_o > 0$.
Suppose   the induced metric $g$ on $ \So$  has positive Gauss curvature.
 Let $ \eta (g)$ be the scaling invariant of $(\So, g)$ stated  in Theorem \ref{thm-main-1}. 
If 
\bee 
 \mathcal{W} : =  \frac{1}{16 \pi} \int_{\Sigma_o} H_o^2 d \sigma  < \eta (g), 
 \eee
then
\be \label{eq-main-1-1}
\left[    \frac{    \mathcal{W} } {  \eta (g) -   \mathcal{W} }  \right]^\frac12  \sqrt{ \frac{ | \Sigma_o  |}{16 \pi} }
+ \mh (\So)  \ge 0 . 
\ee
\end{thm}

\vh 

The quantity  $ \eta(g)$ measures how far $g$ is different from a round metric on $ \So$.
We will give its precise definition  in Section \ref{sec-eta}.
For now  we  give a few  remarks  on  Theorems \ref{thm-main-1} and \ref{thm-main-1-1}.

\begin{remark} \label{rem-eta-shi-tam}
For a fixed $ \delta \in (0,1)$,
it is proved in Proposition \ref{prop-est-eta} that
\be \label{eq-est-eta-intro}
\eta (g) \ge \frac{C}{ || g - g_o ||^{2}_{C^{0, \delta} (\So) } }
\ee
for some positive constant $ C $ independent on $g$ if $ g $ is $C^{2,\delta} $-close to a   round metric $g_o$  on $ \So$.
In particular, $ \eta (g)$ tends to $\infty$ as  $g$ approaches  $g_o$ in the $C^{2, \delta}$-norm.
On the other hand, 
given an $ \Omega$   in Theorem \ref{thm-main-1}, 
 by  Shi and Tam's result \cite[Theorem 1]{ShiTam02} (or more precisely by their proof), one has
 \bee
 \int_{\So} H_o d \sigma < \int_{\So} H_{_E} d \sigma ,
 \eee
 where $H_{_E}$ is the mean curvature of the isometric embedding of $ \So$ in  $ \R^3$.
 Consequently,
 \bee 
 \mathcal{W} < \omega (g) : = \frac{1}{16 \pi | \So |}  \left(  \int_{\So} H_{_E} d \sigma \right)^2 .
 \eee
Therefore,  the condition  $ \mathcal{W} < \eta (g) $ is automatically satisfied  
if   $ \omega (g) \le  \eta (g) $.  By  \eqref{eq-est-eta-intro},  
this  is true  if $ g$ is $C^{2,\delta}$-close to a round metric. 
\end{remark}

\begin{remark}
Given an $ \Omega$   in Theorem \ref{thm-main-1-1}, 
 one knows  $ \mathcal{W} < \eta (g) $ always holds  if 
$g$ is $C^{2,\delta}$-close  to a round metric for the reason explained in Remark \ref{rem-eta-shi-tam}.
Therefore,  inequality \eqref{eq-main-1-1} is  true 
for any CMC surface $\Sigma$ bounding a compact $3$-manifold with nonnegative scalar curvature, provided 
the induced metric on $\Sigma$ is sufficiently round. This may be compared with the  
result of Christodoulou and Yau \cite{C-Y}  
which gives $ \mh (\Sigma ) \ge 0 $ under the 
extrinsic curvature condition.
%stability  assumption. 
\end{remark}

\begin{remark}
On an asymptotically flat $3$-manifold $M$,
there exist foliations  by CMC spheres near infinity
(cf. \cite{HY96, Ye96, Met07, Hua10, EM12, Nerz-CMC-AF}).
For instance,   Nerz \cite{Nerz-CMC-AF}  obtained    the existence and uniqueness of such a foliation
   without assuming asymptotic symmetry conditions. 
Let  $ \{ \Sigma_\sigma \}_{\sigma > \sigma_0} $ be  a foliation of CMC spheres near infinity of $M$  and
suppose $\p M$ consists of outermost minimal surfaces.
Let $ \Omega_\sigma $ be the region  bounded by $ \Sigma_\sigma $ and $ \p M$.
Let $ g_\sigma$ be the induced metric on $ \Sigma_\sigma$.
If $M$ is $C^{2, \delta}_{\tau}$-asymptotically flat with decay rate $\tau > \frac12$,
 it follows from Nerz's work (cf. \cite[Proposition 4.4]{Nerz-CMC-AF})
 that, upon pulling-back to $S^2$, the rescaled metric $\tilde g_\sigma : = \sigma^{-2} g_\sigma $
satisfies\footnote{We  thank Christopher Nerz  for explaining 
this estimate  along the CMC foliation.}
\bee  \label{eq-Nerz}
|| \tilde g_\sigma - g_* ||_{C^{2,\delta} (S^2) }  \le C \sigma^{-\tau}
\eee
for some fixed round metric $g_*$  of area $4\pi$ and  a  constant $C$  independent on $\sigma$.
Thus, along  $\{ \Sigma_\sigma \}$,
$ \mathcal{W}  = 1 + O (\sigma^{-\tau} )$
while $  \eta (g_\sigma)  \to \infty $ by \eqref{eq-est-eta-intro}.
Hence, Theorem \ref{thm-main-1} is applicable to $ \Omega_\sigma$ for large $\sigma$.
However,  our estimate  of $\eta (g)$ in \eqref{eq-est-eta-intro}   
is not strong enough to imply 
$   \left[      \frac{    \mathcal{W} } {  \eta ( \tilde g_\sigma ) -   \mathcal{W} }  \right]^\frac12 
\sqrt{ \frac{ | \Sigma_\sigma  |}{16 \pi} }   \to 0$
along $ \{ \Sigma_\sigma \}$. If this  could be shown, then one would recover 
the Riemannian Penrose inequality by  taking limit of \eqref{eq-main-1}  
since the Hawking mass $ \mh (\Sigma_\sigma)$ approaches to 
the ADM mass  \cite{ADM} along $ \{ \Sigma_\sigma \}$. 
\end{remark}

When $ \p \Omega = \So \cup \Sh $, 
we have another result  separate from Theorem \ref{thm-main-1}. 

\begin{thm} \label{thm-main-2}
Let $ \Omega$ be a compact, orientable, Riemannian $3$-manifold with boundary $ \p \Omega$.
Suppose $ \p \Omega$ is  the disjoint union of  $ \Sigma_o $ 
and $\Sh$ such that 
\begin{itemize}
\item[(a)] $\Sigma_o$ is  a topological $2$-sphere with constant mean curvature $H_o > 0$; 
\item[(b)] $\Sh$, which may have multiple components,  is a minimal surface; and
\item[(c)] there are no other closed minimal surfaces in $\Omega$.
\end{itemize}
Suppose  $\Omega $ has nonnegative scalar curvature and  the induced metric $g$ on $ \Sigma_o$ 
 has positive Gauss curvature.
 There exist  constants  $ 0 <  \beta_g  \le 1 $ and $\alpha_g  \ge 0 $,  determined   by $(\So, g)$,
such that if
\bee 
\mathcal{W}  \coloneqq \frac{1}{16 \pi} \int_{\Sigma_o} H_o^2 d \sigma  < \frac{ \beta_g }{1 + \alpha_g  } ,
\eee
then
\be \label{eq-main-2} 
\sqrt{ \frac{ | \Sh |}{16 \pi} } \le  
\left[ \left(  \frac{   \alpha_g  \mathcal{W} } {  \beta_g -   ( 1 + \alpha_g ) \mathcal{W} }  \right)^\frac12
+ 1 \right]   \mh (\So) .
\ee
If $g$ is a round metric,  one can take $ \beta_g  = 1 $ and $ \alpha_g  = 0 $. In this case, \eqref{eq-main-2} reduces to
$ \sqrt{ \frac{ | \Sh |}{16 \pi} }  \le  \mh (\Sigma_o)  $.
\end{thm}

\begin{remark}
Similar to $ \eta(g)$, the constants 
$ \alpha_g$ and $ \beta_g$ also measure how far $g$ is different from a round metric.
By the proof of Proposition \ref{prop-est-eta} in Section \ref{sec-eta},  
one can take $\alpha_g \to 0$ and $\beta_g \to 1$ as $g$ approaches a round metric. 
As a result, suppose $\Omega$ is normalized so that $ | \So | = 4 \pi $ and the mean curvature
constant $H_o$ satisfies $H_o < 2$, then the condition $ \mathcal{W} < \frac{\beta_g}{1 + \alpha_g} $ is always 
met if $g$ is sufficiently round. 
\end{remark}

Now we outline the idea of the proof of Theorems  \ref{thm-main-1} -- \ref{thm-main-2}.
When the intrinsic metric $g$ on $ \So$ is round, Theorems  \ref{thm-main-1} and \ref{thm-main-2}
follow from \cite{Miao09} and Theorem \ref{thm-main-1-1} follows from \cite{Miao02, ShiTam02}. 
Thus, the major  case to prove   is  when  $g$ is not a round metric.
In this case, our proof is inspired by the work of Mantoulidis-Schoen \cite{M-S}. 
Suppose $(\So , g)$ is not isometric to a round sphere, 
  we  want to  construct a collar extension $(N, \gamma) $ of $ \Omega$, where $N = [0,1] \times \So$ and $\gamma$ is a
  suitably chosen  metric,  such that
  \begin{itemize}
  \item[a)] $\gamma$ has nonnegative scalar curvature;
  \item[b)]  the induced metric from $\gamma$ on $ \Sigma_0 :=  \{ 0 \} \times \So $ agrees with $ g$, and the mean curvature of
 $ \Sigma_0 $ in $(N, \gamma)$ equals the mean curvature $H_o$ of $ \So $ in $\Omega$; and
  \item[c)]    the induced metric from $\gamma$ on $ \Sigma_1: = \{ 1 \} \times \So $ is a round metric, and
   the Hawking mass of $\Sigma_1 $ in $(N, \gamma)$ is suitably controlled by the pair $(g, H_o)$.
  \end{itemize}
We then attach $(N, \gamma)$ to $ \Omega$ (see figure \ref{fig:MS})
to obtain a manifold $\hat \Omega$
whose (outer) boundary $ \Sigma_1$ is a round sphere with constant mean curvature.
Though $ \hat  \Omega$ may not be smooth across $ \So$, conditions a) and b) above ensure that
 the result in \cite{Miao09}, which itself was proved using the Riemannian Penrose inequality \cite{Bray01, H-I01}
can be applied to $\hat  \Omega$ to obtain
\be \label{eq-L-RPI}
 \mh (\Sigma_1) \ge \sqrt{ \frac{ | \Sh |}{16 \pi} } . 
\ee
(If $ \Sh = \emptyset $, we apply the positive mass theorem \cite{Schoen-Yau79, Witten81} instead to
have $ \mh (\Sigma_1 ) \ge 0 $.)
This,   combined with c),  then implies  the inequalities in Theorems   \ref{thm-main-1} -- \ref{thm-main-2}.

 %---------------
%! Figure

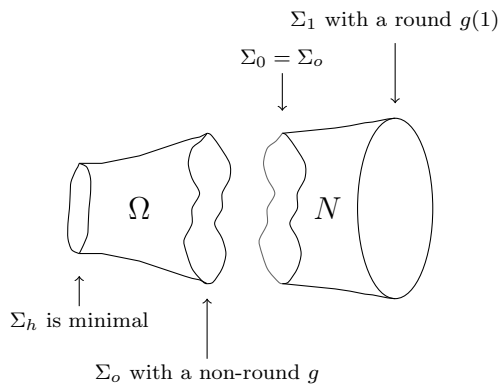
\begin{figure}[ht!] \label{fig:MS}
\begin{tikzpicture}
  \begin{scope}[shift={(-4.2,0)}] % the minimal surface part
      \draw [smooth, black!100] plot coordinates {(0, 0.6) (0.1, 0.55)  (0.15, 0.4) (0.15, 0.2) (0.13, 0)  
    (0.12, -0.2) (0.11, -0.35) (0.05, -0.5)    (0, -0.6)};
     \draw [smooth, black!100] plot coordinates {(0, -0.6) (-0.1, -0.55)  (-0.15, -0.4) (-0.15, -0.2) 
     (-0.13, 0)  
    (-0.12, 0.2) (-0.11, 0.35) (-0.05, 0.5)  (-0.02, 0.55)   (0, 0.6)};
     \node at (0, - 1.5) {\tiny $\Sigma_h$ is minimal}; 
    \draw [black!100, ->] plot coordinates {(0, -1.3) (0, - 0.9)};    
  \end{scope}
  \begin{scope} % \Omega
      \draw [black!100] plot coordinates {(-4.2, 0.6)  (- 3.9, 0.6)   (-3.4, 0.7)  (-2.8, 0.9)  
     (-2.7, 0.93)  (-2.6, 0.95)  (-2.5, 1) };
    \draw [black!100] plot coordinates {(-4.2, -0.6) (- 3.9, -0.6)   (-3.4, -0.7)  (-2.8, -0.9)  
     (-2.7, -0.93)  (-2.6, -0.95)  (-2.5, -1) };
   \begin{scope}[shift={(-2.5,0)}] % \Sigma_o
    \draw [smooth, black!100] plot coordinates {(0, 1) (0.05, 0.98) (0.15, 0.85) (0.3, 0.6) (0.3, 0.4) (0.2, 0.2) (0.2, 0.0) (0.1, -0.2) (0.2, -0.4) (0.25, -0.6) (0.2, -0.8) (0.08, -0.98) (0, -1)};   
    \node at (- 0.9 , 0 ) {$\Omega$};
    \draw [smooth, black!100] plot coordinates {(0, -1) (-0.05, -0.98) (-0.2, -0.8) (-0.3, -0.5) (-0.3, -0.4) (-0.2, -0.2) (-0.2, 0.0) (-0.1, 0.2) (-0.2, 0.4) (-0.25, 0.6) (-0.2, +0.8) (-0.08, +0.9) (0, 1)};
            \node at (0, - 2.2) {\tiny $\Sigma_o$ with \linebreak  a non-round $g$}; 
    \draw [black!100, ->] plot coordinates {(0, -2 ) (0, -1.2)}; 
  \end{scope}
  \end{scope}

  \begin{scope}[shift={(-1.5,0)}] %\Sigma_0
    \draw [smooth, black!100] plot coordinates {(0, 1) (0.05, 0.98) (0.15, 0.85) (0.3, 0.6) (0.3, 0.4) (0.2, 0.2) (0.2, 0.0) (0.1, -0.2) (0.2, -0.4) (0.25, -0.6) (0.2, -0.8) (0.08, -0.98) (0, -1)};
    \draw [smooth, black!60] plot coordinates {(0, -1) (-0.05, -0.98) (-0.2, -0.8) (-0.3, -0.5) (-0.3, -0.4) (-0.2, -0.2) (-0.2, 0.0) (-0.1, 0.2) (-0.2, 0.4) (-0.25, 0.6) (-0.2, +0.8) (-0.08, +0.9) (0, 1)};
        \node at (0, 2) {\tiny $\Sigma_0  = \Sigma_o $ }; 
    \draw [black!100, ->] plot coordinates {(0, 1.8) (0, 1.3)};
  \end{scope}

  \begin{scope} 
    \draw [black!100] plot coordinates {(-1.5, 1)  (-1.4, 1.01)  (-1.3, 1.02)  (-1.2, 1.03)  (-1, 1.05) 
    (-0.8, 1.07) (-0.6, 1.1)  (-0.5, 1.12)   (-0.4, 1.14)  (-0.3, 1.15)  (-0.2, 1.17)  (-0.1, 1.19) (0, 1.2)    };
\node at (- 0.9 , 0 ) {$N$};
    \draw [black!100] plot coordinates {(-1.5, -1)  (-1.4, -1.01)  (-1.3, -1.02)  (-1.2, -1.03)  (-1, -1.05) 
    (-0.8, -1.07) (-0.6, -1.1)  (-0.5, -1.12)   (-0.4, -1.14)  (-0.3, -1.15)  (-0.2, -1.17)  (-0.1, -1.19) 
    (0, -1.2)    };
    \draw [black!100](0, 0) ellipse (0.5 and 1.2); %\Sigma_1
    \node at (0, 2.5) {\tiny $\Sigma_1$ with a round $g(1)$}; 
    \draw [black!100, ->] plot coordinates {(0, 2.2) (0, 1.4)}; 
  \end{scope}
        
\end{tikzpicture}
\caption{A neck $N$ is attached to $\Omega$.}
\end{figure}

%----------------

In  the construction of $(N, \gamma)$, conditions on  $ \mathcal{W}$  are imposed so that
$\gamma$  has nonnegative scalar curvature and the introduction of  $\eta(g)$, $\alpha_g$ and $\beta_g$ 
makes  use of   results  from  \cite{M-S}.

\begin{remark}
It is worth mentioning  that the method described above  indeed reveals information of  the boundary 
component $\So$ in  the non-CMC case as well. Without assuming that $ \So$ is a CMC surface, 
Theorems \ref{thm-main-1} -- \ref{thm-main-2}  remain true if one  let $H_o  = \min_{\So} H $
in  the expressions of $\mathcal{W}$ and   $\mh (\So)$. 
With such a choice of $H_o$, the mean curvature of $\So$ in $\Omega$, which is $H$, dominates 
the  mean curvature of $\Sigma_0$ in $(N, \gamma)$ which is the constant $H_o$ (cf. figure \ref{fig:MS} above). 
Therefore, by employing the techniques in \cite{Miao02}, one knows \eqref{eq-L-RPI} (or $\mh(\Sigma_1) \ge 0 $) 
still holds on  $\hat \Omega$.
\end{remark}

This paper is organized as follows. In Section \ref{sec-collar},  we construct a suitable collar extension of $ \So $.
In Section \ref{sec-app}, we combine the collar extension and the Riemannian Penrose inequality
(or the Riemannian positive mass theorem)
to  draw conclusions  on  $\p \Omega$.
In Section \ref{sec-eta}, we give  the definition and  estimate of $\eta (g)$
and prove Theorems \ref{thm-main-1} -- \ref{thm-main-2}. 
A comparison between inequalities \eqref{eq-main-1} and \eqref{eq-main-2}  is included  in  
an appendix. 

\vspace{.3cm}

\noindent {\em Acknowledgements.}  
The work of PM was partially supported by Simons Foundation Collaboration Grant for Mathematicians \#281105.
 The work of XN  was partially supported by the National Science Foundation of China \#11671089, \#11421061.

\section{Collar extensions} \label{sec-collar}
In this section, we let $ \{ g(t) \}_{t \in [0,1]}$ be a fixed, smooth  path of metrics on $ \Sigma = S^2$, satisfying
\be \label{eq-gauss-K}
 K (g(t)) > 0  ,
 \ee
where $ K ( \cdot ) $ denotes the Gauss curvature of  a metric, and
\be \label{eq-trace-f}
 \tr_{g (t)} g'(t) = 0
 \ee
for all $ t \in [0,1]$, where  $ \tr_{ g(t) } (\cdot) $ is taking  trace  on  $(\Sigma, g(t))$.
Let $ | \Sigma |_{ g(t) }$ be the area of $(\Sigma, g(t))$ which is  a constant  by \eqref{eq-trace-f}.
Let $ r_o > 0 $ be the corresponding constant given by
\be \label{eq-df-ro}
 | \Sigma |_{g(t) } = 4 \pi r_o^2 .
 \ee

We will be interested in a metric $\gamma$ on $ N =  [0,1] \times \Sigma$ of the  form
\bee
\gamma =  A^2 d t^2 + E(t) g(t),
\eee
where $ A > 0 $ is a constant and $ E (t) > 0 $ is a function.
To make a suitable choice of $E(t)$, we consider
part of a spatial Schwarzschild metric
\be \label{eq-S-metric}
\gamma_m = \frac{1}{1 - \frac{2 m}{r} } d r^2 + r^2 g_*
\ee
of mass
$ m \le \frac12 r_o$
defined on $[r_o , \infty) \times S^2$.
Here $g_*$ denotes the standard metric on $S^2$ of area $ 4 \pi$.
We  emphasize that we do allow $m $ to be negative in \eqref{eq-S-metric}.

Making a change of variable
$$  s = \int_{r_o}^r \left( 1 - \frac{2m}{r} \right)^{-\frac12} d r,  $$
we rewrite $\gamma_m$ as
\bee
\gamma_m = d s^2 + u_m^2 (s) g_* ,
\eee
where $ s \in [0, \infty)$ and $ u_m (s) = r(s) $ which  satisfies
\be \label{eq-um-p}
u_m (0) = r_o, \
 \ u_m'(s) = \left(1 - \frac{2m}{ u_m (s) } \right)^\frac12, \
u_m''(s) = \frac{ m}{ u_m(s)^{2} }  .
\ee
Given any constants $ A > 0 $ and $ k \ge 0 $, we  define
\be \label{eq-Et-choice}
E (t) = r_o^{-2} u_m^2 ( A k t ) .
\ee
With such a choice of $ E(t)$, the mean curvature $H(t)$ of
$ \Sigma_t : = \{ t \} \times \Sigma$ with respect to $\gamma$ is
\be \label{eq-H-t}
\begin{split}
H (t) = & \ A^{-1} E^{-1} E' \\
 = & \ 2 k u_m^{-1} \left( 1 - \frac{2m}{u_m} \right)^\frac12
\end{split}
\ee
by \eqref{eq-trace-f} and \eqref{eq-um-p}.
The Hawking mass, $ \mh  (\Sigma_t)$, of $ \Sigma_t$ in $(N, \gamma)$ is
\be  \label{eq-Hmass-1}
\begin{split}
\mh  (\Sigma_t) = & \ \sqrt{ \frac{ | \Sigma_t |_{h (t)} }{ 16 \pi} } \left[ 1 - \frac{1}{16 \pi}
\int_{ \Sigma_t  } H(t)^2 d \sigma_{h (t)} \right] \\
= & \ \frac12 u_m (A k t )  ( 1 -   k^2 ) + m k^2 ,
\end{split}
\ee
where $ h(t) : = E(t) g(t)$ and  $d \sigma_{h(t)}$ is the area element on $(\Sigma_t, h(t))$.

Next we consider the scalar curvature of $\gamma$, denoted by $ R(\gamma)$.
Direct calculation gives
\bee
\begin{split}
R(\gamma)&= 2 K (h)  +A^{-2}\left[-\tr_h h''-\frac{1}{4}(\tr_h h')^2+\frac{3}{4}|{h}'|^2_h\right] ,
\end{split}
\eee
where, by \eqref{eq-trace-f},
\bee
\tr_h h' = 2 E^{-1}E' ,
\eee
\bee
| h' |_h^2  = E^{-2} \left[2 (E')^2 + E^2  |g'|_g^2 \right] ,
\eee
\bee
\tr_h h'' =  2 E^{-1} E'' + \tr_g g''   ,
\eee
and
\bee
0 = [ (\tr_g g') ]'   =   \tr_g g'' - | g'|_g^2  .
\eee
Hence,
\be \label{eq-R-gamma-2}
\begin{split}
R(\gamma)
= & \ E^{-1} 2 K ( g)   +A^{-2}\left[- \frac14 | g' |_g^2   - 2 E^{-1} E''  + \frac{1}{2 }  E^{-2}   (E')^2   \right]  . \\
\end{split}
\ee
Plugging in $ E(t) = r_o^{-2} u_m^2 ( Ak t  )$ and using \eqref{eq-um-p}, we have
\be \label{eq-shwarz}
\begin{split}
& \ A^{-2}\left[   - 2 E^{-1} E''  + \frac{1}{2 }  E^{-2}   (E')^2   \right]  \\
= & \  k^2 \left[   - 2 u_m^{-2}   (u_m')^2  -  4 u_m^{-1}     u_m''  \right]  \\
= & \  k^2 \left[  -  2  u_m^{-2}   \left( 1 - \frac{2 m}{u_m } \right)  - 4 u_m^{-3} m \right] \\
= & \ - k^2  2 u_m^{-2} .
\end{split}
\ee
Therefore, it follows from  \eqref{eq-R-gamma-2} and \eqref{eq-shwarz} that
\be \label{eq-R-gamma-3}
\begin{split}
R(\gamma) = & \  r_o^2 u_m^{-2}  2 K ( g )   - k^2 2 u_m^{-2} - \frac14  A^{-2} | g'|_g^2 \\
= & \ 2 u_m^{-2}  \left[  r_o^{2} K (g) - k^2  - u_m^2 A^{-2}  \frac18 | g'|_g^2  \right]   .
\end{split}
\ee

Now we  define  two quantities  associated to the path $\{ g(t) \}_{t \in [0,1]}$:
\be \label{eq-df-beta}
\beta : = \min_{ t \in [0,1], x \in \Sigma} r_o^2 K ( g(t) ) (x) 
\ee
and
\be \label{eq-df-alpha}
\alpha :  = \max_{ t \in [0,1], x \in \Sigma} \frac{1}{4} | g'|^2_g ( t, x) .
\ee
Clearly, $ \alpha = 0 $ if and only if $\{g(t)\}_{ t \in [0,1] } $
is a constant path.
Moreover,  by the Gauss-Bonnet theorem and \eqref{eq-df-ro},
\be \label{eq-GB-beta-0}
\int_\Sigma r_o^2 K (g(t) ) d \sigma_{g(t)} =  4\pi r_o^2
= \int_\Sigma 1 d \sigma_{g(t)} , \ \forall \ t .
\ee
Therefore,
\be \label{eq-GB-beta}
 \beta \le 1,  \ \mathrm{and}  \
\beta = 1 \Longleftrightarrow
\  r_o^2 K (g(t) ) (x)  = 1, \ \forall \ t, x .
\ee

In terms of  $\beta $ and $\alpha$,
it follows from \eqref{eq-R-gamma-3}  that
\be \label{eq-R-gamma-4}
\begin{split}
R(\gamma) \ge  & \ 2 u_m^{-2}  \left[ \beta - k^2  - \frac12  u_m^2 A^{-2} \alpha   \right]   .
\end{split}
\ee
To further estimate $ R(\gamma)$,  we consider the cases of $ m < 0 $ and $ m \ge 0 $ separately.

\vh

\noindent {\bf Case 1: $m < 0$}.  In this case, \eqref{eq-um-p} and the fact $ u_m (s) \ge r_o $ imply
\be
u_m'(s) \le \left( 1 - \frac{2m}{r_o} \right)^\frac12 ,
\ee
and therefore
\be \label{eq-um-u-bd}
u_m (s) \le r_o +  \left( 1 - \frac{2m}{r_o} \right)^\frac12  s .
\ee
Hence, by \eqref{eq-R-gamma-4} and \eqref{eq-um-u-bd},
\be \label{eq-R-gamma-5}
\begin{split}
R(\gamma) \ge  & \ 2 u_m^{-2}  \left\{ \beta - k^2  - \frac12 \left[ \left( 1 - \frac{2m}{r_o} \right)^\frac12kt + r_o A^{-1} \right]^2 \alpha   \right\}   \\
\ge  & \ 2 u_m^{-2}  \left\{ \beta - k^2  -   \left[ \left( 1 - \frac{2m}{r_o} \right) k^2  + (r_o A^{-1})^2 \right] \alpha   \right\}   .
\end{split}
\ee

\vh

\noindent {\bf Case 2: $m \ge 0 $}. In this case, \eqref{eq-um-p}  implies
$ u_m'(s) \le 1 $ and
\be \label{eq-um-u-bd-p}
u_m (s)  \le   r_o + s .
\ee
Therefore,  by \eqref{eq-R-gamma-4} and  \eqref{eq-um-u-bd-p},
\be \label{eq-R-gamma-5-p}
\begin{split}
R(\gamma) \ge  & \ 2 u_m^{-2}  \left[ \beta - k^2  - \frac12 \left( kt + r_o A^{-1} \right)^2 \alpha   \right]   \\
\ge  & \ 2 u_m^{-2}  \left[ \beta - k^2  -  \left( k^2  + r_o^2  A^{-2 } \right)  \alpha   \right]   .
\end{split}
\ee

\vh

We are led to the following proposition.

\begin{prop} \label{prop-collar}
Given a smooth path of metrics $ \{ g(t) \}_{t \in [0 ,1] } $ on $\Sigma $ satisfying \eqref{eq-gauss-K} and \eqref{eq-trace-f},
let $ r_o $,  $ \beta$ and $ \alpha$ be the constants
defined by \eqref{eq-df-ro}, \eqref{eq-df-beta} and \eqref{eq-df-alpha}, respectively.
Suppose $\alpha > 0 $, i.e. $ \{ g(t) \}_{t \in [0 ,1] } $ is not a constant path.
Let  $ m \le \frac12 r_o $ and $ k \ge  0 $ be two  constants satisfying
\be \label{eq-m-k-cond}
\beta - \left[ 1 + \left( 1 - \frac{2m}{r_o} \right) \alpha \right] k^2 > 0 , \ \mathrm{if} \ m < 0
\ee
or
\be \label{eq-k-cond-p}
\beta - ( 1 + \alpha) k^2 >  0 , \ \mathrm{if} \ m \ge 0 .
\ee
Let  $ A_o  > 0$  be the constant given by
\be \label{eq-df-Ao}
A_o = r_o \left[  \frac{ \alpha} { \beta - \left[ 1 + \left( 1 - \frac{2m}{r_o} \right) \alpha \right] k^2 } \right]^\frac12
\ \mathrm{if} \ m < 0
\ee
or
\be \label{eq-df-Ao-p}
A_o = r_o \left[ \frac{ \alpha} { \beta - ( 1 + \alpha) k^2} \right]^\frac12 , \ \mathrm{if} \ m \ge 0 .
\ee
 Let $ u_m (s)$ be the function defined by \eqref{eq-um-p}.
 Then,  for any constant  $A \ge A_o $,   the metric
 \be
 \gamma = A^2 d t^2 + r_o^{-2} u_m^2 ( A k t  ) g (t)
 \ee
 on $ N =  [0,1] \times \Sigma$ satisfies
 \begin{enumerate}
 \item[(i)]  $ R(\gamma) \ge 0 $, where $ R(\gamma)$ is the scalar curvature of $ \gamma$;
 \item[(ii)] the induced metric  on $\Sigma_0 : = \{ 0 \} \times \Sigma$ is $g(0)$, and  the mean curvature  of $\Sigma_0$
 is  $ H(0) = 2 k r_o^{-1} ( 1 - \frac{2m}{r_o} )^\frac12  $; and
 \item[(iii)]   $ \Sigma_t : = \{ t \} \times \Sigma$  has positive constant mean curvature for each $t$ and its  Hawking mass is
 \begin{equation*}
 \begin{split}
  \mh  ( \Sigma_t)   = & \ \frac12 \left[ u_m (A k t) - r_o \right] ( 1 - k^2) + \mh (\Sigma_0 ) .
  \end{split}
  \end{equation*}
 \end{enumerate}
\end{prop}

\begin{proof}
(i) is a direct corollary of \eqref{eq-R-gamma-5} and  \eqref{eq-R-gamma-5-p}.
(ii) follows from  \eqref{eq-H-t} and  the fact $ u_m (0) = r_o $.
 (iii) is implied by    \eqref{eq-H-t} and \eqref{eq-Hmass-1}.
\end{proof}

\begin{remark}  \label{rem-p-R}
In Proposition \ref{prop-collar},  one indeed  has  $ R(\gamma) > 0 $ on $ [0, 1) \times \Sigma $.
This is because in both \eqref{eq-R-gamma-5} and \eqref{eq-R-gamma-5-p},
 the second inequality is a strict inequality unless $t = 1$.
 Now suppose $ g(1) $ is a round metric and $g(0)$ is not round, then
 $ r_o^2 K (g(1)) = 1 $ and $ \beta < 1 $ by  \eqref{eq-GB-beta}.
Thus, by \eqref{eq-R-gamma-3}, the inequality in \eqref{eq-R-gamma-4} is strict at $t=1$.
Therefore,  in this case,  $ R(\gamma) > 0 $ everywhere on $N$.
\end{remark}

\begin{remark}
When $\alpha = 0 $, by \eqref{eq-R-gamma-4},
 it suffices to require $\beta \ge k^2$
 for $\gamma$ to have $ R(\gamma) \ge 0$.
 In particular, if $\{ g(t)\}_{t \in [0,1]} $ consists of a fixed round metric
 and  $ k^2 = \beta = 1$,
 then $\gamma$ reduces   to the Schwarzschild metric $\gamma_m$.
\end{remark}

\section{Application}  \label{sec-app}

In this section, we let
$ \Omega $ be a  compact Riemannian $3$-manifold
with  the following properties:
\begin{itemize}
\item $\Omega$ has nonnegative scalar curvature;

\item  $\p \Omega$ is the disjoint union of  $\So $ and $ \Sh$, where
 $ \So$ is  a topological  $2$-sphere and $ \Sh$, if nonempty,  is the unique, 
 closed minimal surface  (possibly disconnected)  in $ \Omega$;

 \item the mean curvature of $ \So$ in $\Omega$ is a positive constant $H_o$; and

 \item
  there exists   a smooth path of metrics $ \{ g(t) \}_{t \in [0,1]}$ on $ \Sigma : = \So  $
satisfying \eqref{eq-gauss-K} and \eqref{eq-trace-f}
such that $ g(0) = g $, which is  the induced metric on $ \Sigma $ from $\Omega$, and $ g(1) $ is a round metric.
 \end{itemize}

We will apply a suitable collar extension constructed in Proposition \ref{prop-collar}
and the Riemannian Penrose inequality (or the positive mass theorem) to draw 
information on the geometry of $ \So$.

First, we consider a result
obtained by applying Proposition \ref{prop-collar}
 with parameters $m < 0$.
In this case, we  impose a condition
\be \label{eq-H-condition-m-n}
\left( \frac14 H_o^2 r_o^2  \right) \alpha
< \beta 
\ee
on $ \So$, where $ r_o$ is the area radius of $(\So, g)$ and  $ \beta$, $ \alpha$ are the constants, 
associated to the path $\{ g(t) \}_{t \in [0,1]}$, defined in \eqref{eq-df-beta}, \eqref{eq-df-alpha}, respectively.

\begin{thm} \label{thm-section-m-n}
If  \eqref{eq-H-condition-m-n} holds, then
\be \label{eq--section-m-n}
\frac12 r_o
\left[     \frac{  \frac14 H_o^2  r_o^2  \alpha  } {   \beta-  \frac14 H_o^2  r_o^2 \alpha  }   \right]^\frac12
+ \mh (\So) \ge \sqrt{ \frac{ | \Sh |}{16 \pi} } .
\ee
\end{thm}

\begin{proof}
If $\alpha = 0 $, then $g$ is a round metric. In this case,
 the  claim reduces to $  \mh (\So) \ge \sqrt{ \frac{ | \Sh |}{16 \pi} } $,
 which follows from \cite[Theorem 1]{Miao09}.
 Therefore, it suffices to consider the case $g$ is not round, i.e. $ \alpha > 0 $.

We will construct a suitable metric $\gamma$ on  $N = \Sigma \times [0, 1]$ and attach $(N, \gamma)$ to
$\Omega$ along $  \So $.
To do so,
note that \eqref{eq-H-condition-m-n} implies
there are constants   $ m < 0 $
satisfying
\be \label{eq-choice-of-m-n}
\beta -  \frac14 H_o^2  r_o^2 \alpha   -  \frac14 H_o^2  r_o^2  \left( 1 - \frac{2m }{r_o}   \right)^{-1} > 0 .
\ee
For any such an $m$, define
\be  \label{eq-df-k-app}
k =  \frac12 H_o  r_o \left(1 - \frac{2m }{r_o}  \right)^{-\frac12} .
\ee
Then \eqref{eq-choice-of-m-n} gives
\be  \label{eq-cond-k-m-n}
\beta - \left[ 1 + \left( 1 - \frac{2m}{r_o} \right) \alpha \right] k^2 > 0 .
\ee
Now let
\be \label{eq-df-A-o-app}
A_o = r_o \left[  \frac{ \alpha} { \beta - \left[ 1 + \left( 1 - \frac{2m}{r_o} \right) \alpha \right] k^2 } \right]^\frac12
\ee
and consider the metric
\be
\gamma = A_o^2 dt^2 + r_o^{-2} u_m^2 ( A_o k t) g(t)
\ee
on $N$.
Let $ \Sigma_t : = \{ t \} \times \Sigma$.
It  follows from \eqref{eq-cond-k-m-n}, \eqref{eq-df-A-o-app} and Proposition \ref{prop-collar} that
$(N, \gamma)$ has nonnegative scalar curvature,
each $ \Sigma_t$ has positive constant mean curvature,
the induced metric from $\gamma$ on $  \Sigma_0 $
agrees with $g$,  the mean curvature  $H(0)$  of $\Sigma_0$
equals $H_o$, and the Hawking mass of  $\Sigma_1$
in $(N, \gamma)$ and the Hawing mass  of $\So$ in $\Omega$ are related by
\be \label{eq-H-mass-o-1}
\begin{split}
 \mh  (\Sigma_1) =  & \ \frac12  \left[ u_m (A_o k )  - r_o \right] ( 1 - k^2) + \mh  (\Sigma_0)  .
 \end{split}
 \ee

Now we glue $(N, \gamma)$ and $ \Omega$ along their common boundary component $ \Sigma_0 = \So $
to obtain a Riemannian manifold $\hat {\Omega}$.
The metric $\hat g$ on $ \hOmega$  is Lipschitz across $\So$ and smooth everywhere else;
it has nonnegative scalar curvature away from $\So$; and the mean curvature of $ \So$
from both sides in $\hOmega$ agree.
Moreover,  $ \p \hOmega = \Sh \cup \Sigma_1$ where $ \Sigma_1$  is isometric to a round sphere
and has constant mean curvature.
Therefore, applying the mollification method used in \cite{Miao02, Miao09} which
smooths  out the corner of $\hat g $ at $\So$,
we  know that  \cite[Theorem 1]{Miao09} applies to   $\hat \Omega$ to give
\be \label{eq-H-mass-h-1}
\mh  (\Sigma_1) \ge \sqrt{ \frac{ | \Sh |}{16 \pi} } .
\ee
(A more precise and direct way to derive  \eqref{eq-H-mass-h-1} is as follows.
Since $ \Sigma_1 $ is both round and having constant mean curvature,
we can again attach to $ \hat \Omega$, along $ \Sigma_1$,
a manifold
$ N_\infty = \left( [ r_1, \infty) \times S^2 , \gamma_m \right) $  with
 $ 4 \pi r_1^2  =  | \Sigma_1 |$, $\gamma_m$ given by \eqref{eq-S-metric} and $ m =  \mh  (\Sigma_1)$.
 Indeed,  $N_\infty$ is
the region that is  exterior to a rotationally symmetric sphere with area $ | \Sigma_1 |$
in the spatial Schwarzschild manifold whose mass is $  \mh  (\Sigma_1)$.
We denote the resulting manifold by $\hat M$, which consists of three pieces $\Omega$, $N$ and $N_\infty$.
The metric on $\hat M$ satisfies  the mean curvature matching condition across  both $\So$ and $ \Sigma_1 $.
Therefore, one can repeat the same proof in \cite{Miao09}, starting from Lemma 3 on page 278 and ending at equation (47)
on page 280,  to conclude that the Riemannian Penrose inequality still holds on such an $\hat M$,
which proves   \eqref{eq-H-mass-h-1}.)

To proceed, we note that \eqref{eq-H-mass-o-1} and \eqref{eq-H-mass-h-1} imply
\be \label{eq-H-mass-h-1-more}
\frac12  \left[ u_m (A_o k )  - r_o \right] ( 1 - k^2) + \mh  (\So)  \ge \sqrt{ \frac{ | \Sh |}{16 \pi} } .
 \ee
By \eqref{eq-cond-k-m-n} and  \eqref{eq-GB-beta},
 \be \label{eq-k-l-1}
 k^2 < \beta \leq 1 ,
 \ee
and,  by \eqref{eq-um-u-bd},
\be \label{eq-upp-u-m-r-n}
\begin{split}
u_m (A_o k ) - r_o \le & \ \left( 1 - \frac{2m}{r_o} \right)^\frac12  A_o k \\
= & \  \frac12 H_o r_o  A_o .
\end{split}
\ee
Therefore,  \eqref{eq-H-mass-h-1-more} -- \eqref{eq-upp-u-m-r-n} imply
\be \label{eq-H-mass-h-1-more-1}
 \frac14 H_o r_o  A_o    ( 1 - k^2) + \mh  (\So)  \ge \sqrt{ \frac{ | \Sh |}{16 \pi} } ,
 \ee
 where
 \be
 \begin{split}
  \frac14 H_o r_o  A_o        
  = & \  \frac12   r_o \left[  \frac{  \frac14 H_o^2 r_o^2    \alpha} {  \left( \beta -   \frac14 H_o^2 r_o^2   \alpha \right)  - k^2  } \right]^\frac12  .
 \end{split}
 \ee
In summary, we have proved
\be \label{eq-g-result-m-n}
\frac12   r_o \left[  \frac{  \frac14 H_o^2 r_o^2    \alpha} {  \left( \beta -   \frac14 H_o^2 r_o^2   \alpha \right)  - k^2  } \right]^\frac12
 (1 - k^2) +  \mh  (\So)  \ge \sqrt{ \frac{ | \Sh |}{16 \pi} }
\ee
 for any $ m < 0 $ satisfying \eqref{eq-choice-of-m-n}.

 To obtain a  result that does not involve $m$ or $k$, we can
let $ m \to -\infty$ and   \eqref{eq-df-k-app} shows
\be \label{eq-limit-1}
 \lim_{ m \rightarrow - \infty} k = 0 .
\ee
It follows from \eqref{eq-g-result-m-n} and \eqref{eq-limit-1} that
\be \label{eq-conclusion-1}
\frac12 r_o
\left[    \frac{    \frac14 H_o^2  r_o^2 \alpha } {  \beta-   \frac14 H_o^2  r_o^2  \alpha   }  \right]^\frac12
+ \mh (\So) \ge \sqrt{ \frac{ | \Sh |}{16 \pi} } ,
\ee
which proves the theorem.
\end{proof}

\begin{remark} \label{rem-empty-sh}
If $ \Sh = \emptyset$, i.e  if  $ \Omega$ is merely a compact $3$-manifold with nonnegative
scalar curvature, with boundary $ \p \Omega =  \So $,
 then, replacing the Riemannian Penrose inequality by the Riemannian positive mass theorem
 in the proof,   one has $ \mh(\Sigma_1) \ge 0 $ (cf. \cite{Miao02, ShiTam02}). 
 %instead of   \eqref{eq-H-mass-h-1}.
 In this case, the result becomes
\be
\frac12 r_o
\left[    \frac{    \frac14 H_o^2  r_o^2 \alpha  } {   \beta-   \frac14 H_o^2  r_o^2  \alpha  }  \right]^\frac12
+ \mh (\So) \ge 0 .
\ee
\end{remark}

Next, we consider a corresponding result obtained by
applying Proposition \ref{prop-collar} with parameters $m \ge 0$.
In this case, we  assume  a condition
\be \label{eq-H-condition-m-p}
\frac14 H_o^2 r_o^2 < \frac{ \beta}{ 1 + \alpha }.
\ee

\begin{thm} \label{thm-section-m-p}
Suppose \eqref{eq-H-condition-m-p} holds. Given any constant $ m \in \left[0, \frac12 r_o \right) $
satisfying
\be \label{eq-eq-H-condition-m-p-p}
\frac14 H_o^2 r_o^2
< \frac{ \beta}{ 1 + \alpha }  \left( 1 - \frac{2m }{r_o}   \right),
\ee
define
\be  \label{eq-df-k-app-m-p}
k =  \frac12 H_o  r_o \left(1 - \frac{2m }{r_o}  \right)^{-\frac12} , \
A_o = r_o \left[  \frac{ \alpha} { \beta - \left( 1 +  \alpha \right) k^2 } \right]^\frac12  .
\ee
Then
\bee
\frac12  A_o k ( 1 - k^2)  + \mh (\So) \ge \sqrt{ \frac{ | \Sh |}{16 \pi} } .
\eee
In particular, if one chooses $m = 0 $, then
\be \label{eq-m-equal-0}
  \left[  \frac{ \alpha \left(  \frac14 H_o^2  r_o^2 \right) } { \beta  - \left( 1 +  \alpha \right) \left(  \frac14 H_o^2  r_o^2 \right) } \right]^\frac12   \mh (\So)    + \mh (\So) \ge \sqrt{ \frac{ | \Sh |}{16 \pi} } ,
\ee
and consequently 
\be \label{eq-m-equal-0-r}
  \left[  \frac{  \frac14 H_o^2  r_o^2  } { \frac{\beta}{\left( 1 +  \alpha \right) }   -  \frac14 H_o^2  r_o^2  } \right]^\frac12   \mh (\So)    + \mh (\So) \ge \sqrt{ \frac{ | \Sh |}{16 \pi} }  . 
\ee
\end{thm}

\begin{proof}
Again, it suffices to assume  $ \alpha > 0 $.
By \eqref{eq-eq-H-condition-m-p-p} and \eqref{eq-df-k-app-m-p},
\be \label{eq-b-k-m-app-m-p}
\begin{split}
  & \ \beta - \left( 1 +  \alpha \right) k^2 \\
= &  \  \beta - \left( 1 + \alpha \right) \frac14 H_o^2  r_o^2  \left( 1 - \frac{2m }{r_o}   \right)^{-1} \\
> & \ 0  .
\end{split}
\ee
Consider the metric
$$
\gamma = A_o^2 dt^2 + r_o^{-2} u_m^2 ( A_o k t) g(t)
$$
on $N = [0, 1] \times \Sigma$.
Let $ \Sigma_t : = \{ t \} \times \Sigma$.
It  follows from \eqref{eq-df-k-app-m-p},  \eqref{eq-b-k-m-app-m-p} and Proposition \ref{prop-collar} that
$(N, \gamma)$ has nonnegative scalar curvature,
the induced metric from $\gamma$ on $  \Sigma_0 $
agrees with $g$,  the mean curvature  $H(0)$  of $\Sigma_0$
equals $H_o$, and the Hawking mass of  $\Sigma_1$
in $(N, \gamma)$ and the Hawing mass  of $\So$ in $\Omega$ are related by
\be \label{eq-H-mass-o-1-m-p}
\begin{split}
 \mh  (\Sigma_1)
 = & \   \frac12  \left[ u_m (A_o k )  - r_o \right] ( 1 - k^2) + \mh  (\So)  .
 \end{split}
 \ee
Attaching $(N, \gamma)$ to $ \Omega$,
we  have
\be \label{eq-H-o-mass-h-1-m-p}
\mh  (\Sigma_1) \ge \sqrt{ \frac{ | \Sh |}{16 \pi} }
\ee
by the reason explained in the proof of Theorem \ref{thm-section-m-n}.
It follows from \eqref{eq-H-mass-o-1-m-p}  and \eqref{eq-H-o-mass-h-1-m-p}  that
\be \label{eq-H-mass-h-1-more-m-p}
\frac12  \left[ u_m (A_o k )  - r_o \right] ( 1 - k^2) + \mh  (\So)  \ge \sqrt{ \frac{ | \Sh |}{16 \pi} } .
 \ee
 Again, since $ \beta \le 1 $, \eqref{eq-b-k-m-app-m-p} implies $ k^2 <   1 $.
 Also,   \eqref{eq-um-u-bd-p} shows
 $$ u_m (A_o k ) - r_o \le  A_o k .  $$
Therefore,  \eqref{eq-H-mass-h-1-more-m-p} implies
\be \label{eq-H-mass-h-1-more-1-m-p}
\frac12   A_o k  ( 1 - k^2) + \mh  (\So)  \ge \sqrt{ \frac{ | \Sh |}{16 \pi} } ,
 \ee
 where
 \be
 A_o k  = r_o \left[  \frac{ \alpha k^2 } { \beta  - \left( 1 +  \alpha \right) k^2 } \right]^\frac12  .
\ee
Thus, we have proved
\be \label{eq-g-result-m-p}
\frac12  r_o \left[  \frac{ \alpha k^2 } { \beta  - \left( 1 +  \alpha \right) k^2 } \right]^\frac12 ( 1 - k^2)
+ \mh  (\So)  \ge \sqrt{ \frac{ | \Sh |}{16 \pi} }
\ee
for any $ m \in [ 0, \frac12 r_o)  $ satisfying \eqref{eq-eq-H-condition-m-p-p}.

To obtain a  result that does not involve $m$ or $k$, we can  take  $ m = 0 $. In this case,
$ k = \frac12 H_o r_o$ and  \eqref{eq-g-result-m-p} becomes
 \be
 \begin{split}
 \left[  \frac{ \alpha \frac14 H_o^2 r_o^2 } { \beta - \left( 1 +  \alpha \right) \frac14 H_o^2 r_o^2  } \right]^\frac12
   \mh (\So)   + \mh  (\So)  \ge \sqrt{ \frac{ | \Sh |}{16 \pi} }   ,
 \end{split}
 \ee
 which proves \eqref{eq-m-equal-0}. Inequality \eqref{eq-m-equal-0-r} follows from  \eqref{eq-m-equal-0} simply 
 by the fact $\frac{\alpha}{1 + \alpha} \le 1 $. This completes the proof. 
 \end{proof}

\begin{remark} \label{rem-minimizing-m}
In the derivation of Theorems \ref{thm-section-m-n} and  \ref{thm-section-m-p},
besides taking $ m = - \infty$ and $ m =0$,
one can  minimize the first term in \eqref{eq-g-result-m-n} and \eqref{eq-g-result-m-p},
subject to the constraint
$m $ satisfies \eqref{eq-choice-of-m-n} and
\eqref{eq-eq-H-condition-m-p-p}, respectively.
We leave this calculation   in Appendix \ref{appen}.
\end{remark}

\begin{remark}
If $g$ is not a round metric, i.e. $\alpha > 0 $, the collar $(N, \gamma)$ that we attached to $\Omega$
indeed has strictly positive scalar curvature by Remark \ref{rem-p-R}. Therefore,
by the rigidity statement of the Riemannian Penrose inequality, one naturally would expect
that inequalities in \eqref{eq-H-mass-h-1-more} and \eqref{eq-H-mass-h-1-more-m-p}
are indeed  strict. Therefore, equalities in Theorems \ref{thm-section-m-n} and  \ref{thm-section-m-p}
should  hold only if $ \alpha = 0$, i.e. when $g$ is a round metric on $\So$.
However, we do not have a rigorous proof of this claim.
\end{remark}

\section{Definition of $\eta(g)$} \label{sec-eta}
In this section, we  define  the quantity $\eta (g)$  and prove Theorems \ref{thm-main-1} -- \ref{thm-main-2}.
Given a metric $g$ with positive Gauss curvature on $ \Sigma = S^2$,
 let $\{ h(t) \}_{t \in [0,1]}$ denote a smooth path of metrics on $ \Sigma$ such that
\begin{itemize}
\item[(i)] $h(0) $ is isometric to $ g $ and $h(1)$ is a round metric;
\item[(ii)]
 $h(t) $ has positive Gauss curvature, i.e. $K( h(t) ) > 0 $, $\forall \ t$; and
\item[(iii')]  $ | \Sigma |_{h(t)} = | \Sigma |_{g}$, i.e. the area of $ (\Sigma, h(t)) $  is a constant,  $\forall \ t$.
\end{itemize}
There are various ways to construct such a path.
For instance, one may apply  the uniformization theorem
to write $g = e^{2 w} g_o$ for some  function
$w$ and a  round metric $g_o$, and
to define $h(t) = e^{ 2 (1- t) w} g_o$ (cf. \cite{Nirenberg}),  followed by an area normalization.

Given such a path $\{ h(t) \}_{t \in [0,1]}$,
 applying the  proof of Lemma 1.2 in \cite{M-S} to $\{ h(t) \}_{ t \in [0,1]}$,
one  can  construct   a new  path of metrics $\{ g(t) \}_{ t \in [0,1]}$,
satisfying  (i)  and (ii), with $h(t)$ replaced by $g(t)$, together with  the following property  that is stronger than (iii'):
\begin{itemize}
\item[(iii)] $ \frac{d}{dt} d \sigma_{g(t)} = 0 $,  or equivalently $ \tr_{g(t)} g'(t) = 0 $,  $\forall \ t $.
Here  $ d \sigma_{g(t)}$ is the area form of $g(t)$.
\end{itemize}
We include this construction of  $\{ g(t) \}_{ t \in [0,1]}$ by Mantoulidis and Schoen  in the  lemma below
for the purpose of later obtaining estimates on $\eta(g)$.

\begin{lemma}[\cite{M-S}]  \label{lem-ht-2-gt}
Given $\{ h(t) \}_{t \in [0,1]} $  satisfying (i), (ii) and (iii') above,
there exists $\{ g(t) \}_{t \in [0,1]} $ satisfying (i), (ii) and (iii).
\end{lemma}
\begin{proof}
Let $ \nabla_{h(t)}$,  $ \Delta_{ h(t)}$ denote the gradient, the Laplacian  on $(\Sigma, h(t))$, respectively.
Given a $1$-parameter family of diffeomorphisms $\{ \phi_t \} $ on $ \Sigma$, define
$ g(t) : = \phi_t^* (h(t) ) $.
Then
\be  \label{eq-gpt}
g'(t) = \phi_t^* ( h'(t)) + \phi_t^* \left( L_X  h(t)  \right),
\ee
\be
\tr_{g(t)} g'(t) = \phi_t^* \left( \tr_{h(t)} \left( h'(t) + L_X h(t) \right) \right) ,
\ee
where $ X = X(x, t)$ is the vector field  satisfying $ \frac{d}{dt} \phi_t  = X ( \phi_t , t) $
and $L$ denotes the Lie derivative on $\Sigma$.
Thus, to satisfy (iii), it suffices to demand
$  \tr_{h(t)} L_X h(t)   =  - \tr_{ h(t) } h'(t) $, i.e.
\be \label{eq-X}
\div_{h(t)} X = - \frac12 \tr_{ h(t) } h'(t)  .
\ee
A way to pick such an $X$ is to assume $ X  = \nabla_{h(t)} u $ for some function $u = u(x,t)$
satisfying
 \be \label{eq-u}
 \Delta_{h(t)} u  =   - \frac12 \tr_{ h(t) } h'(t)  \ \mathrm{and} \ \  \int_{\Sigma} u \ d \sigma_{h(t)} = 0 .
 \ee
Since
$$ \int_\Sigma \tr_{ h(t) } h'(t)  d \sigma_{h(t)} = 0 $$
by (iii'),
  \eqref{eq-u} has a unique solution $u$ that depends smoothly on $t$ whenever $h(t)$ is smooth on $t$.
This finishes the proof.
\end{proof}

Given any smooth  path  $\{ g(t) \}_{ t \in [0,1] }$ with properties (i), (ii) and (iii),
%as in \eqref{eq-df-beta} and \eqref{eq-df-alpha}, 
let
\bee \label{eq-df-beta-intro}
\beta_{ \{ g(t) \} } : = \min_{ t \in [0,1], x \in \Sigma} \frac{1}{4\pi} | \Sigma  |_{g(t)}  K ( g(t) ) (x)  
\eee
and
\bee \label{eq-df-alpha-intro}
\alpha_{ \{ g(t) \} }   :  = \max_{ t \in [0,1], x \in \Sigma} \frac{1}{4} | g'|^2_g ( t, x)  , 
\eee
where $ | g'|^2_g $ denotes  the square  norm of $g' (t)$ with respect to $g(t)$.
\begin{definition} \label{df-eta}
Given a metric $g$ with positive Gauss curvature on $\Sigma = S^2$, define
\bee
\eta (g) : = \sup_{ \{ g(t) \} } \frac{ \beta_{ \{ g(t) \} } }{ \alpha_{ \{ g(t) \} } }  ,
%\ \ \mathrm{and} \ \ 
%\kappa (g) : = \sup_{ \{ g(t) \} } \frac{ \beta_{ \{ g(t) \} } }{ 1 +  \alpha_{ \{ g(t) \} } }  ,
\eee
where the supremum is taken over all paths  $\{ g(t) \}_{ t \in [0,1] }$
satisfying (i), (ii) and (iii).
Similarly, one may also define 
$$ 
\kappa (g) : = \sup_{ \{ g(t) \} } \frac{ \beta_{ \{ g(t) \} } }{ 1 +  \alpha_{ \{ g(t) \} } }  .
$$
\end{definition}
Clearly, $\eta (g)$ and $ \kappa (g)$ satisfy 
\bee
 0 < \eta (g)  \le \infty \ \ \mathrm{and} \ \ 
0 < \kappa (g)  \le 1 ,
\eee
where the second inequality follows from  \eqref{eq-GB-beta}. Moreover, 
for constant $c>0$, it is  straightforward to check that
\be
  \eta ( c^2 g) = \eta (g)  \ \ \mathrm{and} \ \  \kappa ( c^2 g) = \kappa (g) . 
  \ee
If $g = g_o$ is a round metric, 
by taking  $\{ g(t) \}$ to be a constant path,  one has  $ \alpha_{ \{ g(t) \} } = 0 $
and $\beta_{ \{ g(t) \} } = 1 $,  hence  
\be
 \eta (g_o) = \infty \ \ \mathrm{and} \ \   \kappa (g_o) = 1 . 
 \ee
Below, we give   a lower bound of $\eta (g)$ and $ \kappa (g)$  for  $g$ that is close to a round metric.

\begin{prop}  \label{prop-est-eta}
Let $ g_*$ be  the standard metric of area $4 \pi$ on $ \Sigma = S^2$.
There exists a constant  $ \epsilon_0 > 0 $ such that if
$
 || g - g_* ||_{C^{2, \delta} (\Sigma)} < \epsilon_0 ,
 $
then
\be \label{eq-est-eta}
 \eta (g) \ge \frac{C}{ || g - g_* ||^2_{C^{0, \delta} (\Sigma)  } } 
 \ \ \mathrm{and} \ \ 
 \kappa (g) \ge  1 - C  || g - g_*   ||_{C^{2,\delta} (\Sigma) }  .
 \ee
Here $C$ is some positive  constant that is independent on $g$ and
$ || \cdot ||_{C^{k, \delta} (\Sigma) } $ is the $C^{k, \delta}$ norm on $(\Sigma, g_*)$ for an integer $k \ge 0$
and a constant  $ \delta \in (0,1)$.
\end{prop}

\begin{proof}
Given any $ \epsilon > 0 $, let $ U_\epsilon $ be  the set of metrics $g$  satisfying 
$ || g - g_* ||_{C^{2, \delta} (\Sigma)} < \epsilon $.
First, choose a small $\epsilon_0   $ so that elements  in $ U_{\epsilon_0} $ all have  positive Gauss curvature.

Given any  $g \in U_{\epsilon_0}$, let  $ \tau = g - g_*$. 
Then  $ || \tau ||_{C^{2,\delta}(\Sigma) } < \epsilon_0 $.
For  each $ t \in [0,1]$,
define $ \tilde h (t)$, $ a(t)$ and $ h(t)$, respectively  by 
\be \label{eq-df-of-tht}
\tilde h (t) = g_* + ( 1 - t ) \tau , \ \ 
| \Sigma |_{\tilde h (t)} =  a(t) | \Sigma |_g , \  \
h(t) = a^{-1}(t) \tilde h(t) .
\ee
Then
$ | \Sigma |_{h(t)} = a^{-1}(t) | \Sigma |_{ \tilde h (t)} = | \Sigma |_g $.
Hence, $ \{ h(t) \}_{t \in [ 0,1] } $ is a path satisfying properties (i), (ii) and (iii').
Moreover,
\be \label{eq-tht-g*}
|| \tilde h (t) - g_* ||_{C^{2, \delta} (\Sigma) } \le  ||  \tau ||_{C^{2,\delta} (\Sigma) } , \ \ 
| a(t) - 1 | \le C_1 || \tau ||_{C^{2, \delta} (\Sigma)} ,    
\ee
and
\be \label{eq-ht-g*}
\begin{split}
& \ || h(t) - g_* ||_{C^{2,\delta} (\Sigma) } \\
 = & \  || a^{-1}(t) ( 1 - t ) \tau + ( a^{-1} (t) - 1 ) g_* ||_{C^{2, \delta} (\Sigma)}  \\
\le & \  C_2 ||  \tau ||_{C^{2,\delta} (\Sigma) } .
\end{split}
\ee
Here and below, $C_1$, $C_2$, $ ... $  always  denote  constants that  do not depend  on $ \tau $ and $t$.

Now  let $ \{ g (t ) \}_{ t \in [0,1] } $ be the path of metrics constructed from $\{ h(t) \}_{ t \in [0,1] }$ in the proof of Lemma \ref{lem-ht-2-gt}.
It follows from  \eqref{eq-ht-g*} and the fact $ g(t) = \phi_t^* ( h(t) ) $ that
\be \label{eq-est-beta}
\beta_{ \{ g(t) \} } = \frac{ | \Sigma |_g} { 4 \pi}  \min_{ t \in [0,1],  x \in \Sigma} K ( h(t) ) (x)  \ge 1 - C_3 ||  \tau ||_{C^{2,\delta} (\Sigma) } .
\ee
We next estimate $ \alpha_{ \{ g(t) \} }$.
 By \eqref{eq-gpt}, 
 $ g'(t) = \phi_t^* ( H(t) ), $
where
$$ H(t) = h'(t) + L_X h(t)  .$$
Hence,
$
| g' |^2_g = \phi_t^* ( | H |_{h}^2 ) .
$
Therefore,
\be \label{eq-alpha-est-1}
\begin{split}
\alpha_{ \{ g(t) \} } =  & \  \max_{ t \in [0,1], x \in \Sigma} \frac{1}{4} | H |^2_h ( t, x)  \\
\le & \  \max_{ t \in [0,1], x \in \Sigma} \frac{1}{2}  \left[ | h' |_h^2 + | L_X h(t) |_h^2 \right]  ( t, x)  .
\end{split}
\ee
Plugging  in $ X = \nabla_{h(t) }u$, we have
\be \label{eq-LX-h-1}
L_X h (t) = 2 \nabla^2_{h(t)} u,
\ee
where $ \nabla^2_{h(t)} $ denotes the Hessian  on $(\Sigma, h(t))$.
By \eqref{eq-u}, \eqref{eq-ht-g*} and the standard linear elliptic estimates,
we have
\be \label{eq-schauder}
|| u ||_{C^{2, \delta} (\Sigma )} \le C_4   || \tr_{ h(t) } h'(t) ||_{C^{0, \delta} (\Sigma) } .
\ee
Therefore, by  \eqref{eq-LX-h-1} and \eqref{eq-schauder},
\be \label{eq-LXh-2}
| L_X h(t) |_{h} \le C_5     || \tr_{ h(t) } h'(t) ||_{C^{0, \delta}(\Sigma) } .
\ee
It follows from \eqref{eq-alpha-est-1} and \eqref{eq-LXh-2} that
\be \label{eq-alpha-3}
\alpha_{\{ g(t) \} }  \le \max_{ t \in [0,1], x \in \Sigma} \frac{1}{2}   | h' |_h^2 (t, x) +
\max_{ t \in [0,1]}  C_6   || \tr_{ h(t) } h'(t) ||^2_{C^{0, \delta}(\Sigma) }  .
\ee
By \eqref{eq-df-of-tht}, we have
\be \label{eq-ht-hpt}
\tr_{ h(t) } h'(t) =  -  2 a^{-1} a' -  \tr_{ \tilde h(t) } \tau ,
\ee
\be
\begin{split}
| h' |_{h}^2  = & \  2 a^{-2} (a')^2  +  | \tau |_{\tilde h}^2 + 2 a^{-1} a' \tr_{\tilde h(t)} \tau ,
\end{split}
\ee
\be \label{eq-apt}
 a'(t)    = - \frac{1}{2 | \Sigma |_g }   \int_{\Sigma}  \tr_{\tilde h (t) } \tau d \sigma_{\tilde h(t)}  .
\ee
Thus,   by  \eqref{eq-tht-g*} and   \eqref{eq-ht-hpt} -- \eqref{eq-apt},  we have 
\be \label{eq-final-1}
 | h' |_{h}^2  \le C_7  || \tau ||^2_{C^0(\Sigma)} \ \ \mathrm{and} \ \ 
|| \tr_{ h(t) } h'(t)  ||^2_{C^{0, \delta} (\Sigma)} \le C_8 || \tau  ||^2_{C^{0, \delta} (\Sigma)} .
\ee
Finally, by \eqref{eq-alpha-3} and  \eqref{eq-final-1}, we conclude
\be \label{eq-est-alpha}
\alpha_{ \{ g(t) \} } \le C_9 || \tau ||^2_{C^{0, \delta} (\Sigma)}  .
\ee
Estimate \eqref{eq-est-eta}  then  follow readily from  \eqref{eq-est-beta} and  \eqref{eq-est-alpha}. 
\end{proof}

We now give  the proof  of Theorems \ref{thm-main-1} -- \ref{thm-main-2}.

\begin{proof}[Proof of Theorems \ref{thm-main-1} and \ref{thm-main-1-1}]
It suffices to assume that  $g$ is not a round metric.
Let  $ \{ g^{(j)}(t) \}_{ t \in [0,1] }$, $ j = 1, 2, ... $, be a sequence of path of metrics,  satisfying (i), (ii) and (iii), such that
$$
\frac{ \beta_{ \{ g^{(j)} (t) \} } }{ \alpha_{ \{ g^{(j)} (t)  \} } }   \rightarrow \eta (g), \ \mathrm{as}  \ j \to \infty.
$$
Suppose $ \mathcal{W} < \eta (g)$, then
$$ \mathcal{W} <  \frac{ \beta_{ \{ g^{(j)} (t) \} } }{ \alpha_{ \{ g^{(j)} (t)  \} } }, \ \mathrm{for \ large} \  j . $$
For these $j$,  by Theorem \ref{thm-section-m-n} and Remark \ref{rem-empty-sh}, 
\bee
\frac12 r_o
\left[    \frac{    \mathcal{W} } { { \alpha_{ \{ g^{(j)} (t)  \} }}^{-1} \beta_{ \{ g^{(j)} (t)  \} }-   \mathcal{W}  }  \right]^\frac12
+ \mh (\So) \ge \sqrt{ \frac{ | \Sh |}{16 \pi} } , \  \ \mathrm{when} \ \Sh \neq \emptyset 
\eee
and
\bee
\frac12 r_o
\left[    \frac{    \mathcal{W} } { { \alpha_{ \{ g^{(j)} (t)  \} }}^{-1} \beta_{ \{ g^{(j)} (t)  \} }-   \mathcal{W}  }  \right]^\frac12
+ \mh (\So) \ge  0  , \  \ \mathrm{when} \ \p \Omega = \So.
\eee
Taking $ j \to \infty$, Theorems \ref{thm-main-1} and \ref{thm-main-1-1} follow.
\end{proof}

\begin{proof}[Proof of Theorem \ref{thm-main-2}]
Assume that $g$ is not a round metric. 
Pick  any  path $\{ g(t) \}_{ t \in [0,1] }$ used in Section \ref{sec-app} and
choose $ \alpha_g$, $\beta_g$ to be $ \alpha $, $ \beta$ associated to that path, respectively. 
Theorem \ref{thm-main-2} then  follows directly  from \eqref{eq-m-equal-0}   in  Theorem \ref{thm-section-m-p}.
\end{proof}

It would be desirable to improve  Theorem \ref{thm-main-2} in a way that Theorem \ref{thm-main-1} is 
proved  from Theorem \ref{thm-section-m-n}.  However, due to 
the fact that \eqref{eq-m-equal-0} involves both $ \frac{ \beta}{1 + \alpha} $ and $\frac{\alpha}{1 + \alpha}$,
we can only  replace $ \frac{ \beta}{1 + \alpha} $ by $\kappa(g)$ at the expense of  giving up  $\frac{\alpha}{1 + \alpha}$. 
 We record the following theorem. 

\begin{thm} \label{thm-main-2-1}
Let $ \Omega$ be a compact, orientable, Riemannian $3$-manifold with boundary $ \p \Omega$.
Suppose $ \p \Omega$ is  the disjoint union of  $ \Sigma_o $ 
and $\Sh$ such that 
\begin{itemize}
\item[(a)] $\Sigma_o$ is  a topological $2$-sphere with constant mean curvature $H_o > 0$; 
\item[(b)] $\Sh$, which may have multiple components,  is a minimal surface; and
\item[(c)] there are no other closed minimal surfaces in $\Omega$.
\end{itemize}
Suppose  $\Omega $ has nonnegative scalar curvature and  the induced metric $g$ on $ \Sigma_o$ 
 has positive Gauss curvature.
 Let $ 0 < \kappa(g)  \le 1 $ be the scaling invariant of $(\So, g)$ defined  in Definition \ref{df-eta}. 
If
\bee 
\mathcal{W}  \coloneqq \frac{1}{16 \pi} \int_{\Sigma_o} H_o^2 d \sigma  < \kappa(g) ,
\eee
then
\be \label{eq-main-2-1} 
\sqrt{ \frac{ | \Sh |}{16 \pi} } \le  
\left[ \left(  \frac{  \mathcal{W} } {  \kappa (g)  -   \mathcal{W} }  \right)^\frac12
+ 1 \right]   \mh (\So) .
\ee
\end{thm}
\begin{proof}
If $g$ is round, we have $ \sqrt{ \frac{ | \Sh |}{16 \pi} } \le   \mh (\So) $, in particular  \eqref{eq-main-2-1} holds.
So we assume that $g$ is not a round metric. Similar to the proof of Theorem \ref{thm-main-1} above, 
let  $ \{ g^{(j)}(t) \}_{ t \in [0,1] }$, $ j = 1, 2, ... $, be a sequence of path of metrics,  satisfying (i), (ii) and (iii), with 
$$
\frac{ \beta_{ \{ g^{(j)} (t) \} } }{ 1 + \alpha_{ \{ g^{(j)} (t)  \} } }   \rightarrow \kappa (g), \ \mathrm{as}  \ j \to \infty.
$$
Suppose $ \mathcal{W} < \kappa (g)$, then
$$ \mathcal{W} <  \frac{ \beta_{ \{ g^{(j)} (t) \} } }{ 1 + \alpha_{ \{ g^{(j)} (t)  \} } }, \ \mathrm{for \ large} \  j . $$
For these $j$,  by \eqref{eq-m-equal-0-r} in Theorem \ref{thm-section-m-p}, 
 \be \label{eq-app-j-r}
  \left[  \frac{ \mathcal{W} } 
  {  \frac{ \beta_{ \{ g^{(j)} (t)  \} }  }{ 1 + \alpha_{ \{ g^{(j)} (t)  \} } }  -  \mathcal{W} } \right]^\frac12   \mh (\So)    + \mh (\So) \ge \sqrt{ \frac{ | \Sh |}{16 \pi} } .
\ee
Taking $j \to \infty$, Theorem  \ref{thm-main-2-1} follows.
\end{proof}

To  end  this paper, we remark that, besides employing the construction of  
Mantoulidis and Schoen in Lemma \ref{lem-ht-2-gt},
there are other methods to obtain $\{ g(t) \}_{ t \in [0,1] }$  
satisfying (i), (ii) and (iii) used  in  Definition \ref{df-eta}.
For instance, one may apply Hamilton's modified Ricci flow \cite{Ham} on closed surfaces. 
Using results from \cite{Ham, Chow},
 Lin and Sormani \cite{L-S14} introduced a concept of asphericity mass
 for a CMC  surface normalized to have area $4\pi$ and used it to  obtain upper bounds of the surface's Bartnik mass.
It would be  interesting to understand the relation between $ \eta (g)$ or $\kappa(g)$  and the asphericity mass since 
they are all determined  solely  by the intrinsic metric on the surface.
It is also conceivably possible that the modified Ricci flow \cite{Ham} may be used to obtain refined estimates of $\eta(g)$
and $\kappa(g)$. We leave these for interested readers.

\vspace{.2cm}

\begin{appendices}

\makeatletter
\def\@seccntformat#1{Appendix\ \csname the#1\endcsname\quad}
\makeatother

\section{} \label{appen}
In this appendix, we give  the calculation, stated in  Remark \ref{rem-minimizing-m},  which  minimizes the left side
of  \eqref{eq-g-result-m-n} and \eqref{eq-g-result-m-p},
subject to the condition
$m $ satisfies \eqref{eq-choice-of-m-n} and
\eqref{eq-eq-H-condition-m-p-p}, respectively.

We first consider the context of Theorem \ref{thm-section-m-p}. Suppose $ \alpha > 0 $.
Let $ \W = \frac14 H_o^2 r_o^2 $ and define
\be \label{eq-A-kappa}
\kappa : = \frac{\beta}{1 + \alpha} \in (0, 1) .
\ee
Condition \eqref{eq-H-condition-m-p} becomes
$ 
\W < \kappa
$
and the constraint \eqref{eq-eq-H-condition-m-p-p} is
 \be \label{eq-const-2}
 \mathcal{W} < \kappa \left( 1 - \frac{2m}{r_o} \right), \ m \in [0, \frac12 r_o ) .
 \ee
The quantity that we want to minimize is
 \be
 \begin{split}
\Phi  : = & \  \frac12 r_o  \left[  \frac{ \alpha k^2} { \beta - \left( 1 +  \alpha \right) k^2 } \right]^\frac12 (1 -k^2)   \\
  = & \  \frac12 r_o  \left(  \frac{ \alpha} {  1 +  \alpha   } \right)^\frac12  \left[  \frac{  x } { \kappa -  x  } \right]^\frac12 (1 -x )
 \end{split}
 \ee
 where
 $ x: =  k^2 = \W \left( 1 - \frac{2m}{r_o} \right)^{-1} $. 
In terms of $x$, the constraint \eqref{eq-const-2} translates into
$
\W \le x < \kappa  .
$
The solution to this calculus problem   can be derived by
considering
\be
f (x) : = \left(  \frac{  x } { \kappa  - x  } \right) ( 1 - x)^2  ,
\ee
whose derivative is
$
f' (x) = \frac{ (1 - x)}{ (\kappa  - x )^2 }   \left( 2 x^2 - 3 \kappa x + \kappa   \right) .
$
We therefore  have

\vspace{.2cm}

\noindent {\bf Theorem 3.2'}
{\em
In the setting of Theorem \ref{thm-section-m-p}, suppose $\alpha > 0 $ and let $ \kappa $ be given by
\eqref{eq-A-kappa}. Then
$
\min_{  \W \le x< \kappa} \Phi (x)  + \mh (\So) \ge \sqrt{ \frac{ | \Sh |}{ 16 \pi} } ,
$
where
\begin{itemize}
\item[a)]  if $  \kappa \le \frac{8}{9}  $ or
 if  $  \kappa >  \frac{8}{9}  $ and
$ x_2 : = \frac{ 3 \kappa + \sqrt{ 9 \kappa^2 - 8 \kappa} }{4}  \le \W $,  then
$  \min_{  \W \le x< \kappa} \Phi (x)  = \Phi |_{x = \W}  ;  $
\item[b)]   if  $  \kappa >  \frac{8}{9}  $  and
$ x_1 : = \frac{ 3 \kappa - \sqrt{ 9 \kappa^2 - 8 \kappa} }{4}  \le \W < x_2 $, then
$ \min_{  \W \le x< \kappa} \Phi (x)  = \Phi |_{x = x_2 } ;  $
\item[c)]   if    $  \kappa >  \frac{8}{9}  $ and $ \W < x_1 $, then
$  \min_{  \W \le x< \kappa} \Phi (x)  =  \min \left\{ \Phi |_{x = \W}, \ \Phi |_{x = x_2 } \right\} .  $
In particular, since $ \Phi |_{x = x_2} $ is determined only by  $\alpha$ and $\beta$, 
$  \min_{  \W \le x< \kappa} \Phi (x)  =  \Phi |_{x = \W} $
for small $ \W$.
\end{itemize}
Here   $ x_1, x_2 \in (0, \kappa)$  are the roots to  $2x^2 - 3 \kappa x + \kappa = 0 $, and 
$$  \Phi |_{x = \W} = \Phi |_{m =0}
=  \left[  \frac{ \alpha \frac14 H_o^2 r_o^2 } { \beta - \left( 1 +  \alpha \right) \frac14 H_o^2 r_o^2  } \right]^\frac12
   \mh (\So)  .$$
}

\vspace{.3cm}

Next  we consider the context of Theorem \ref{thm-section-m-n}. Suppose $ \alpha > 0 $.
Define
\be \label{eq-A-b}
b : = \beta - \alpha \W \in (0,1),
\ee
where $\W =  \frac14 H_o^2 r_o^2 $.
The condition \eqref{eq-H-condition-m-n} becomes 
$
b > 0
$
and the constraint \eqref{eq-choice-of-m-n} is 
\be \label{eq-const-n}
b > \W \left( 1 - \frac{2m}{r_o} \right)^{-1} , \ m < 0 .
\ee
The quantity that we  want to minimize is
 \be
 \begin{split}
\Psi  : = & \  \frac12 r_o  \left[  \frac{ \alpha \W } { ( \beta -  \alpha \W) - k^2 } \right]^\frac12 (1 -k^2)   \\
  = & \    \frac12 r_o  \left( \alpha \W \right)^\frac12  \left[  \frac{ 1 } { b  - x } \right]^\frac12 (1 -x )   \\
 \end{split}
 \ee
 where
 $ x: =  k^2 = \W \left( 1 - \frac{2m}{r_o} \right)^{-1} $.
There are two cases to consider when interpreting  the constraint.
If $ b < \W  $,   \eqref{eq-const-n}  translates into
$
0 <   x < b   .
$
If $ \W \le b  $,  \eqref{eq-const-n}  translates into
$
0 <   x < \W    .
$
In either case,
the solution to this calculus problem   can be derived by
considering
\be
\tilde f (x) : = \left(  \frac{  1  } { b  - x  } \right) ( 1 - x)^2  ,
\ee
whose derivative is
$
\tilde f' (x) = \frac{ (1 - x)}{ (b  - x )^2 }   \left[  x  - ( 2 b  - 1)    \right] .
$
We therefore  have

\vspace{.2cm}

\noindent {\bf Theorem 3.1'}
{\em
In the setting of Theorem \ref{thm-section-m-n}, suppose   $\alpha > 0 $ and let $b$ be  given by \eqref{eq-A-b}.
\begin{enumerate}
\item If $ b < \W$, then
$
\min_{ 0 < x< b } \Psi   + \mh (\So) \ge \sqrt{ \frac{ | \Sh |}{ 16 \pi} } ,
$
where
\vh
\begin{itemize}
\item[a)]  if $ b \le  \frac12  $,
$ \min_{ 0 < x< b } \Psi   = \Psi |_{x = 0 +  }  ;  $
\vh
\item[b)]   if  $ b >  \frac12   $,
$ \min_{  0 < x< b } \Psi (x)  = \Psi |_{x = 2b - 1 } .  $
\end{itemize}
\vh
\item If $  \W \le b $, then
$ \min_{ 0 < x< \W  } \Psi   + \mh (\So) \ge \sqrt{ \frac{ | \Sh |}{ 16 \pi} } , $
where
\vh
\begin{itemize}
\item[a)]  if $ b \le  \frac12  $,
$ \min_{ 0 < x< \W  } \Psi   = \Psi |_{x = 0 +  }  ;  $
\vh
\item[b)]   if  $   \frac12 < b < \frac{1 + \W}{2}  $,
$ \min_{  0 < x< \W  } \Psi (x)  = \Psi |_{x = 2b - 1 } ;   $
\vh
\item[c)]   if  $    b \ge  \frac{1 + \W}{2}  $,
$ \min_{  0 < x< \W  } \Psi (x)  = \Psi |_{x = \W -  } .   $
\end{itemize}
\end{enumerate}
Here
$$  \Psi |_{x = 0 + } : = \lim_{x \rightarrow 0+ } \Psi = \lim_{m  \rightarrow - \infty  } \Psi =
\frac12 r_o
\left[    \frac{    \frac14 H_o^2  r_o^2 \alpha  } {   \beta-   \frac14 H_o^2  r_o^2  \alpha  }  \right]^\frac12  $$
and
$$  \Psi |_{x = \W -  } : = \lim_{x \rightarrow \W -  } \Psi = \lim_{m  \rightarrow  0 -   } \Psi
=  \left[  \frac{ \alpha \frac14 H_o^2 r_o^2 } { \beta - \left( 1 +  \alpha \right) \frac14 H_o^2 r_o^2  } \right]^\frac12
   \mh (\So)   .
$$
}

It follows from {Theorem 3.1'}  and {Theorem 3.2'}  (2) that, if
$ \W < \frac{\beta}{1 + \alpha} $,
there are cases, depending on $\W$, $\alpha$ and $\beta$, in which
the optimal values of  $ \Phi$ and $ \Psi$ both occur  at $m=0$ and they agree.
\end{appendices}

\end{document}